\documentclass[reqno]{amsart}
\usepackage{amssymb}
\usepackage[english]{babel}
\usepackage{cmap}
\usepackage[unicode,pdftex,pagebackref]{hyperref}
\hypersetup{pdfstartview={FitH},plainpages=false}

\newtheorem{theorem}{Theorem}
\newtheorem{lemma}[theorem]{Lemma}

\numberwithin{equation}{section}

\newcommand{\unit}{{\mathbf1}{\,}}
\newcommand{\spaceR}{{\mathbb R}}
\newcommand{\spaceD}{{\mathbf D}}
\newcommand{\spaceL}{{\mathbf L}}
\newcommand{\spaceC}{{\mathbf C}}
\newcommand{\spaceAC}{{\mathbf{AC}}}
\newcommand{\spaceLip}{{\spaceL_{\infty}^{p\nu}}}
\newcommand{\spaceLi}{{\spaceL_{\infty}}}
\renewcommand{\le}{\leqslant}
\renewcommand{\ge}{\geqslant}
\newcommand{\dd}{\displaystyle}
\newcommand{\norm}[1]{\|\,#1\|}
\newcommand{\w}{\widetilde}
\newcommand{\cal}{\mathcal}
\newcommand{\calT}{{\cal T}}
\DeclareMathOperator*{\vraisup}{vrai\, sup}

\begin{document}
\title[On the solvability of singular equations]{On the solvability of linear singular functional differential equations}

\author{E.I. Bravyi}

\address{E.I. Bravyi \newline
Department of Mathematics \\
State National Research Po\-ly\-tech\-nic University of Perm\\ 
Komsomol'sky pr., 29, Perm, 614990, Russia}
\email{bravyi@perm.ru}


\begin{abstract}
Necessary and sufficient conditions for the solvability of boundary value problems for a family of func\-ti\-o\-nal dif\-fe\-ren\-ti\-al equations with a non-in\-teg\-ra\-ble singularity are obtained.
\end{abstract}

\maketitle

\centerline{\it Dedicated to Lyna Rakhmatullina on the occasion of her birthday}

\vskip20pt

\section{Introduction}

For the last decade a lot of works were devoted to searching conditions for the solvability for different boundary value problems for linear functional differential equations. Unimprovable effective conditions for the solvability are obtained for many problems. Note the monographs
\cite{Mu2000, HLS2002-5, KigPuzha},
and the articles
\cite{HLP2002-2, 5, Hakl-Mukh-2005, Mu2006-1, Mukh2006,
Mukh2006-1, Mu2007, Hakl2009}
on the solvability of periodic problem,
\cite{Mu1997-2, Mu1997-1, MuSremr2005, MuSremr2006}
on the solvability of two-point problem for second order equations,
\cite{Hakl, BHL2} on the solvability of the Cauchy problem for scalar
equations, and \cite{Sremr2006, Sh2010} on the solvability of the Cauchy problems for systems of equations.

Here we consider
boundary value problems for singular linear functional differential equations in the Volterra and non-Volterra cases.
For the non-Volterra case we obtain unimprovable in a sense conditions for the unique solvability of some boundary value problems (Theorems \ref{ket-1}, \ref{ket-3}, \ref{ket-2}, \ref{ket-4}) in the terms of the norms of positive and negative parts of regular functional operators. These conditions are close to the solvability conditions from
\cite{HLS2002-5, HLP2002-2, 5, Hakl-Mukh-2005, Mu2006-1,Mukh2006,Mukh2006-1,Mu2007,Hakl2009,Mu1997-2,Mu1997-1,
MuSremr2005,MuSremr2006,Hakl, BHL2,Sremr2006,Sh2010}.
For the Volterra case, the similar results on the solvability of the Cauchy problem were obtained in \cite{Zernov-2001-53-4, Zernov-1} for some classes of nonlinear singular functional differential equations.

In \cite{R1, R2, R3} a new natural approach to singular functional differential equations is formulated. The basic of this approach is 1) a special choice of the solution space, which must coincide with the space of all solutions of some ``model'' boundary value problem for a simple ``model'' singular equation, 2) using the Fredholm property of boundary value problems for completely continuous perturbations of the ``model'' equation. Such approach was realized for some kinds of singularities in \cite{Lab, Shind, Bravyi, Kung}.

We use the following notation:
$\spaceR=(-\infty,\infty)$;

$\spaceL_q[a,b]$, $q\ge1$  is the Banach space of measurable functions $z:[a,b]\to\spaceR$ such that 
$$\norm{z}_{\spaceL_q[a,b]}=\left(\int_a^b|z(s)|^q\,ds\right)^{1/q}<+\infty\ \text{ for $q<+\infty$}, $$
$$\norm{z}_{\spaceL_\infty[a,b]}=\vraisup\limits_{s\in[a,b]}|z(s)|<+\infty.$$

$\spaceL_q=\spaceL_q[0,1]$, $\spaceL=\spaceL_1[0,1]$;

$\spaceC[a,b]$ is the Banach space of continuous functions $x:[a,b]\to\spaceR$ with the norm $$\norm{x}_{\spaceC[a,b]}=\max_{t\in[a,b]}|x(t)|;$$

$\spaceC=\spaceC[0,1]$;

$\spaceAC[a,b]$ is the Banach space of absolutely continuous functions $x:[a,b]\to\spaceR$ with the norm $$\norm{x}_{\spaceAC[a,b]}=|x(b)|+\int_a^b|\dot x(s)|\,ds;$$

$\spaceAC=\spaceAC[0,1]$.

Consider a ``model'' \cite[p.~91]{R3} singular equation:
\begin{equation}\label{ke-1}
    \dot x(t)=-\dfrac{k}{t}x(t)+f(t),\quad t\in(0,1],
\end{equation}
where $k\ne0$ is a given real constant, $f\in\spaceL$.

We extend to the case $q=1$ some results from  \cite{Plaksina2010, Plaksina2011, Plaksina2012}, where equation \eqref{ke-1} for $f\in\spaceL_q$, $q>1$ was used as a model one.  Moreover,
new unimprovable in a sense conditions for the solvability of the initial problem for singular functional differential equations will be obtained.

A locally absolutely continuous on $(0,1]$ function $x:(0,1]\to\spaceR$ is called a solution of
\eqref{ke-1} if it satisfies this equation almost everywhere on  $(0,1]$.
Since \eqref{ke-1} is not singular in each interval $[\varepsilon,1]$ ($\varepsilon\in(0,1)$), then any its solution has a representation
\begin{equation}\label{ke-2}
    x(t)=t^{-k}x(1)-\int_t^1\left(\frac{s}{t}\right)^kf(s)\,ds=
    t^{-k}\bigl(x(1)-\int_t^1s^kf(s)\,ds\bigr),\quad t\in(0,1].
\end{equation}
Denote by $\spaceD$ the set of all solutions of equation \eqref{ke-1} for all $f\in\spaceL$.

The Cauchy problem for singular equations and problems with weighted initial equations are considered, in particular, in \cite{Sokhadze-1,
Sokhadze-3, KigSoh-1997, KigSoh-1997-2,Sokhadze-4, Sokhadze-5, Sokhadze-6,
Agarwal, Ronto-10,Ronto-20,Ronto-30}.
In \cite{Sokhadze-1, Sokhadze-3,KigSoh-1997, KigSoh-1997-2,Sokhadze-4, Sokhadze-5, Sokhadze-6}  for nonlinear singular functional differential equations,
the conditions for the solvability of the Cauchy problem and problems with weighted initial conditions were obtained (including the many-dimensional case). In these works, the solvability conditions contain one-sided restriction on the right-hand members of singular equations.

We can interpret the works \cite{Ronto-10, Ronto-20, Ronto-30} as a research  of singular equations in the space of solutions of the ``model'' singular equation
\begin{gather}
    \dot x(t)=p(t)f(t),\quad t\in[0,1],\label{nnew-1}
\end{gather}
for all $f\in\spaceL$,
where $p:(0,1]\to(0,+\infty)$ is a non-increasing continuous function such that $\lim_{t\to0+}p(t)=+\infty$. Any function $x:(0,1]\to\spaceR$ such that $x\in\spaceAC[\varepsilon,1]$ for each $\varepsilon(0,1)$ is called a solution of
\eqref{nnew-1} if it satisfies the equation everywhere on $[0,1]$.
It is easy to show that
for every such solution $x$ there exists the finite limit
$\lim_{t\to0+}x(t)/p(t)$.

The results archived in this work is close to \cite{Sokhadze-1,
Sokhadze-3, KigSoh-1997, KigSoh-1997-2,Sokhadze-4, Sokhadze-5, Sokhadze-6,
Ronto-10,Ronto-20,Ronto-30}, however a special choice of ``model'' equations \eqref{ke-1} and \eqref{ke-55} allows to consider more general singularities.

Further in Section \ref{ss-2} we describe properties of the solution spaces $\spaceD_+$ and $\spaceD_-$ of ``model'' equation
\begin{equation*}
    \dot x(t)=-\frac{k}{t} x(t)+f(t),\quad t\in(0,1],
\end{equation*}
for $k>0$ ($\spaceD_+$) and $k<0$ ($\spaceD_-$). Section \ref{ss-3} deals with the Fredholm property of boundary value problems in the spaces $\spaceD_+$ and $\spaceD_-$, and Section \ref{ss-4}  with the conditions for solvability of the Cauchy problem (Theorems \ref{ket-1} and \ref{ket-3}).
In Section \ref{ss-5} more general ``model'' equation
\begin{equation*}
    \dot x(t)=-p(t) x(t)+f(t),\quad t\in(0,1],
\end{equation*}
is considered, where the function $p:(0,1]\to\spaceR$ is positive, its contraction on every interval $[\varepsilon,1]$ is integrable, and $\lim\limits_{\varepsilon\to0+}\int_\varepsilon^1 p(t)\,dt=\infty$.
In this case we also obtain the condition for the existence of solutions of the Cauchy problem
(Theorems \ref{ket-2} and \ref{ket-4}). In section \ref{ss-6} we prove conditions (\eqref{rr-13}, \eqref{rr-130}) for the existence of a solution of  the Cauchy problem for
the equation
\begin{equation*}
   \frac{1}{p(t)} \dot x(t)=-k x(t)+(Tx)(t)+f(t),\quad t\in(0,1],
\end{equation*}
where  $T:\spaceC\to \spaceL$ is a linear bounded operator.

\section{A space of solutions for a singular equation}\label{ss-2}

Representation \eqref{ke-2} establishes a bijection between the spaces $\spaceD$ and  $\spaceL\times\spaceR$. Therefore, $\spaceD$ is the Banach space with respect the norm
\begin{equation}\label{ke-3}
    \norm{x}_{\spaceD}=|x(1)|+\int_0^1|\dot x(s)+\frac{k}{s}x(s)|\,ds.
\end{equation}

Let $k<0$. Denote $\spaceD_-=\spaceD$.
The embedding operator $x\to x$  acts from $\spaceD_-$ into $\spaceAC$
and bounded. Indeed, let  $x\in\spaceD_-$ be a solution of \eqref{ke-1}. Then
\begin{equation*}
    \norm{x}_{\spaceD_-}=|x(1)|+\int_0^1|f(s)|\,ds,
\end{equation*}
\begin{equation*}
\begin{split}
    \norm{x}_{\spaceAC}=|x(1)|+\int_0^1|\dot x(s)|\,ds=|x(1)|+\int_0^1|f(s)-\frac{k}{s}x(s)|\,ds\le \\
    |x(1)|+\int_0^1|f(s)|\,ds+\int_0^1\frac{|k|}{s}|x(s)|\,ds\le \\
    |x(1)|+\int_0^1|f(s)|\,ds+|k|\int_0^1s^{|k|-1}\,ds\,|x(1)|+|k|\int_0^1 t^{|k|-1}\int_t^1 \frac{|f(s)|}{s^{|k|}}\,ds\,dt= \\
    2|x(1)|+\int_0^1|f(s)|\,ds+|k|\int_0^1 t^{|k|-1}\int_t^1 \frac{|f(s)|}{s^{|k|}}\,ds\,dt.
\end{split}
\end{equation*}
Changing the order of integration in the last integral (here and further in \eqref{ke-13}
it is possible on the Fubini theorem since all integrands are non-negative), we have
\begin{equation*}
\begin{split}
\int_0^1 t^{|k|-1}\int_t^1 \frac{|f(s)|}{s^{|k|}}\,ds\,dt=
\int_0^1 \int_0^s  t^{|k|-1}\, dt \frac{|f(s)|}{s^{|k|}}\,ds=\frac{1}{|k|}\int_0^1|f(s)|\,ds.
\end{split}
\end{equation*}
So,
\begin{equation}\label{ke-6}
\begin{split}
    \norm{x}_{\spaceAC}\le 2|x(1)|+2\int_0^1|f(s)|\,ds\le 2 \norm{x}_{\spaceD_-}.
\end{split}
\end{equation}
Therefore, $\dot x\in\spaceL$ for every  $x\in\spaceD_-$, and there exists a finite limit $x(0+)$. Let  all elements of the space $\spaceD_-$ be extended by continuity at $t=0$. For each $x\in\spaceD_-$ there exists a function $f\in\spaceL$ such that
$$\frac{k}{t}x(t)=f(t)-\dot x(t),\quad t\in[0,1],$$
where the right-hand side is integrable on $[0,1]$. Thus,
$$\int_0^1 \frac{1}{t}|x(t)|\,dt< +\infty,$$
what is possible for the continuous function $x$ only if $x(0)=0$. So,
$$\lim\limits_{t\to0} x(t)=0\text{ for all }x\in\spaceD_-.$$
Inequality \eqref{ke-6} means that the space $\spaceD_-$ is continuously embedded in
$\spaceAC$.

For $k>0$ it follows from \eqref{ke-2} that a solution of \eqref{ke-1} has a finite limit at the point $t=0+$ if and only if
\begin{equation}\label{ke-8}
    x(1)=\int_0^1s^kf(s)\,ds.
\end{equation}
In this case
\begin{equation}\label{ke-9}
    x(t)=\int_0^t\left(\frac{s}{t}\right)^kf(s)\,ds,\ t\in(0,1],
\end{equation}
and from \eqref{ke-9} it follows that
$$|x(t)|\le \int_0^t\left(\frac{s}{t}\right)^k|f(s)|\,ds\le \int_0^t|f(s)|\,ds,\quad t\in(0,1],$$
hence $\lim\limits_{t\to0+} x(t) =0$ by the absolute continuity  of the Lebesgue integral.
For $k>0$ denote by $\spaceD_+$  the space of all solutions of \eqref{ke-1} satisfying  \eqref{ke-8} and extended by continuity at $t=0$. With the norm \eqref{ke-3} this space is Banach, isomorphic to  $\spaceL$:
\begin{equation}\label{ke-10}
    \norm{x}_{\spaceD_+}=|x(1)|+\int_0^1|\dot x(s)+\frac{k}{s}x(s)|\,ds.
\end{equation}
A one-to-one correspondence between $\spaceD_+$ and $\spaceL$ is established by \eqref{ke-9}. The space $\spaceD_+$ is continuously embedded in $\spaceAC$: if $f\in\spaceL$ and $x\in\spaceD_+$ is defined by \eqref{ke-9}, then
\begin{equation*}
    \norm{x}_{\spaceD_+}=|x(1)|+\int_0^1|f(s)|\,ds,
\end{equation*}
\begin{equation*}
\begin{split}
    \norm{x}_{\spaceAC}=|x(1)|+\int_0^1|\dot x(s)|\,ds=|x(1)|+\int_0^1|f(s)-\frac{k}{s}x(s)|\,ds\le \\
    |x(1)|+\int_0^1|f(s)|\,ds+\int_0^1\frac{k}{s}|x(s)|\,ds\le \\
    |x(1)|+\int_0^1|f(s)|\,ds+k\int_0^1 \frac{1}{t^{k+1}}\int_0^t s^k |f(s)|\,ds\,dt.
\end{split}
\end{equation*}
Changing the integration order in the last integral,  we have
\begin{equation}\label{ke-13}
\begin{split}
\int_0^1 \frac{1}{t^{k+1}}\int_0^t s^k |f(s)|\,ds\,dt=
\int_0^1 \int_s^1  \frac{1}{t^{k+1}}\, dt\, s^k|f(s)|\,ds=
\int_0^1  \frac{1-s^k}{k} |f(s)|\,ds\le\\
\le\frac{1}{|k|}\int_0^1|f(s)|\,ds.
\end{split}
\end{equation}
Therefore, the space $\spaceD_+$ is continuously embedded in $\spaceAC$:
\begin{equation*}
    \norm{x}_{\spaceAC}\le |x(1)|+2\int_0^1|f(s)|\,ds\le 2 \norm{x}_{\spaceD_+}.
\end{equation*}

Note, that every set $\spaceD_-$, $\spaceD_+$ is a proper subset of the set of functions from $\spaceAC$ satisfying the condition $x(0)=0$. Otherwise, from \eqref{ke-1} for $x(t)=\int_0^t f(s)\,ds$, $t\in[0,1]$, it follows that for any integrable function $f\in\spaceL$ the function
$$f(t)+\frac{k}{t}\int_0^t f(s)\,ds,\quad t\in(0,1],$$
is also integrable, but it is known that the Cesaro operator
$$(Af)(t)=\frac{1}{t}\int_0^t f(s)\,ds,\quad t\in(0,1],$$
does not act in the space $\spaceL$.

Nevertheless, for any $q>1$ the operator $A$ acts in the space $\spaceL_q$ and bounded. If $q=\infty$, it is obviously, if $1<q<\infty$, it follows from one of generalization of Herdy's inequality (see, for example, \cite[p. 294]{Hardy})
$$\int_0^1\left(\frac{1}{t}\int_0^t |f(s)|\,ds\right)^p\,dt<\left(\frac{p}{p-1}\right)^p\int_0^1|f(s)|^p\,ds.$$

Functional differential equations with the ``model'' equation \eqref{ke-1} and $f$ from $\spaceL_q$ for $q>1$ are investigated in \cite{Plaksina2010, Plaksina2011, Plaksina2012}. In this case the correspond spaces $\spaceD_-$, $\spaceD_+$ match  the space of functions from $\spaceAC$ satisfying $\dot x\in\spaceL_q$ and $x(0)=0$. Our work extends \cite{Plaksina2010, Plaksina2011, Plaksina2012} by considering the case $q=1$.

Now it is obviously that the sets $\spaceD_-$, $\spaceD_+$ do not depend on $k\ne0$ and coincide:
$$\{x|x\in\spaceD_-\}=\{x|x\in\spaceD_+\}.$$
For any $k$, such a set coincides with the set $$\w\spaceAC_0\equiv\{x\in\spaceAC:\ x(0)=0,\  \int_0^1\frac{|x(t)|}{t}\,dt<+\infty\}.$$ Indeed, if $x\in\spaceD_-$ or $x\in\spaceD_+$ for some $k\ne0$, then $x(0)=0$, $x\in\spaceAC$, and the equality
$$\frac{x(t)}{t}=\frac{1}{k}\left(f(t)-\dot x(t) \right), \quad t\in(0,1],$$
holds for some $f\in\spaceL$. It implies that $x\in\w\spaceAC_0$. Conversely, if $x\in\w\spaceAC_0$, then the function
$$f(t)=\dot x(t)+k\frac{x(t)}{t}, \quad t\in(0,1],$$
is integrable for each $k\ne0$, therefore $x$ belongs to every sets $\spaceD_-$ and $\spaceD_+$.

For every positive $k$, the space $\spaceD_+$ is isomorphic to the space $\spaceL$, therefore all norms \eqref{ke-10} on the set $\w\spaceAC_0$ for different $k>0$ are equavivalent. For each negative $k$ the space $\spaceD_-$ is isomorphic to the space $\spaceL\times\spaceR$, where
$$\norm{\{f,c\}}_{\spaceL\times\spaceR}=\norm{f}_{\spaceL}+|c|,\quad f\in\spaceL,\quad c\in\spaceR.$$
It means that all norms \eqref{ke-3} on the set $\w\spaceAC_0$ for different $k<0$ are equivalent. Since there does not exist a linear bounded map between  $\spaceL\times\spaceR$ and $\spaceL$, then the norms in the space $\spaceD_+$ and  $\spaceD_-$ are not equivalent.

It follows from \eqref{ke-2}, that
boundary value problem
\begin{equation}\label{ke-15}
\left\{
\begin{array}{l}
    \dot x(t)=-\frac{k}{t}x(t)+f(t),\quad t\in[0,1],\\
    x(1)=c,
\end{array}
\right.
\end{equation}
is uniquely solvable for $k<0$ in the space $\spaceD_-$ (the solution is defined by \eqref{ke-2}). For $k=0$, problem \eqref{ke-15} is uniquely solvable in the space $\spaceAC$.

From  \eqref{ke-9} it follows that
the equation
\begin{equation}\label{ke-16}
    \dot x(t)=-\frac{k}{t}x(t)+f(t),\quad t\in[0,1],\\
\end{equation}
is uniquely solvable in the space $\spaceD_+$  for $k>0$ (the solution $x\in\spaceD_+$ is defined by \eqref{ke-9} for all $f\in\spaceL$). Note, equation \eqref{ke-16} in the space  $\spaceD_+$ can be written in the equivalent form of the Cauchy problem
\begin{equation}\label{ke-17}
\left\{
\begin{array}{l}
    \dot x(t)=-\frac{k}{t}x(t)+f(t),\quad t\in[0,1],\\
    x(0)=0.
\end{array}
\right.
\end{equation}
Now and for  $k=0$ problem \eqref{ke-17} (for the non-singular equation) is uniquely solvable in the space $\spaceAC$.

\section{The Fredholm property of boundary value problems}\label{ss-3}

Consider the equation
\begin{equation}\label{ke-18}
     \dot x(t)=-\frac{k}{t}x(t)+(Tx)(t)+f(t),\quad t\in(0,1],
\end{equation}
in the space $\spaceD_-$ or $\spaceD_+$ (according to the sign of $k\ne0$),
where  $T:\spaceC[0,1]\to \spaceL[0,1]$ is a linear bounded operator.
A solution of \eqref{ke-18} is a function $x\in\w\spaceAC_0$, for which equality \eqref{ke-18} holds for almost all  $t\in[0,1]$.

Let $\ell:\spaceAC\to\spaceR$ be a linear bounded functional.

Since $\spaceD_-$ is continuously embedded into $\spaceAC$, then for $k<0$
the boundary value problem
\begin{equation}\label{ke-19}
\left\{
\begin{array}{l}
    \dot x(t)=-\dfrac{k}{t}x(t)+(Tx)(t)+f(t),\quad t\in[0,1],\\
    \ell x=c,
\end{array}
\right.
\end{equation}
has the Fredholm property. Let us prove it. The problem
\begin{equation}\label{ke-20}
\left\{
\begin{array}{l}
    \dot x(t)=-\dfrac{k}{t}x(t)+(Tx)(t)+f(t),\quad t\in[0,1],\\
    x(1)=c,
\end{array}
\right.
\end{equation}
has the Fredholm property. Indeed, from \eqref{ke-2} it follows that \eqref{ke-20} is equivalent to the equation in the space $\spaceD_-$:
\begin{equation*}
    x(t)=-\int_t^1\left(\frac{t}{s}\right)^{|k|}(Tx)(s)\,ds +ct^{|k|}-\int_t^1\left(\frac{t}{s}\right)^{|k|}f(s)\,ds,\quad t\in[0,1],
\end{equation*}
which can be rewritten as
\begin{equation}\label{ke-21}
    x(t)=(\Lambda Tx)(t) + g,
\end{equation}
where
$$(\Lambda z)(t)=-\int_t^1\left(\frac{t}{s}\right)^{|k|}z(s)\,ds\quad t\in[0,1],$$
$$g(t)=ct^{|k|}-\int_t^1\left(\frac{t}{s}\right)^{|k|}f(s)\,ds,\quad t\in[0,1],$$
$\Lambda:\spaceL\to\spaceD_-$ is a linear bounded operator, $g\in\spaceD_-$.

Since the space $\spaceD_-$ is continuously embedded in $\spaceAC$, which is continuously embedded in $\spaceC$, then the linear operator $\Lambda$ acts from $\spaceL$ into $\spaceC$ and bounded. So, we can consider \eqref{ke-21} in  $\spaceC$: every solution from $\spaceC$ is a solution from  $\spaceD_-$ and visa versa.

From Lemma 2 in \cite{Bravyi1}, the operator $I-\Lambda T:\spaceC\to\spaceC$
($I:\spaceC\to\spaceC$ is the identical operator) has the Fredholm property.
Every linear bounded operator $T:\spaceC\to\spaceL$ is weakly completely continuous \cite[VI.4.5]{DSh}. Therefore, operator  $\Lambda T:\spaceC\to\spaceC$ is
weakly completely continuous, and the product of such operators, the operator $(\Lambda T)^2:\spaceC\to\spaceC$, is completely continuous \cite[VI.7.5]{DSh}.
By Nikolsky's theorem  \cite[p. 504]{Kantorovich-Akiliv}, the operator $I-\Lambda T$ has the Fredholm property. Thus, problem \eqref{ke-20} and its finite-dimensional perturbation   \eqref{ke-19} have the Fredholm property.

In particular, problem \eqref{ke-19} is uniquely solvable if and only if the homogeneous problem
\begin{equation}\label{ke-23}
\left\{
\begin{array}{l}
    \dot x(t)=-\dfrac{k}{t}x(t)+(Tx)(t),\quad t\in[0,1],\\
    \ell x=0,
\end{array}
\right.
\end{equation}
has only the trivial solution in the space $\spaceD_-$.

Similarly, for $k>0$ using the continuity of the embedding the space $\spaceD_+$ into $\spaceAC$, we prove that the equation
\begin{equation}\label{ke-24}
    \dot x(t)=-\dfrac{k}{t}x(t)+(Tx)(t)+f(t),\quad t\in[0,1],
\end{equation}
or the equivalent Cauchy problem
\begin{equation}\label{ke-25}
\left\{
\begin{array}{l}
    \dot x(t)=-\dfrac{k}{t}x(t)+(Tx)(t)+f(t),\quad t\in[0,1],\\
    x(0)=0,
\end{array}
\right.
\end{equation}
in the space  $\spaceD_+$ has the Fredholm property for every linear bounded operator $T:\spaceC\to\spaceL$. Therefore, \eqref{ke-24} has a unique solution $x\in\spaceD_+$ for every  $f\in\spaceL$ if and only if the equation
\begin{equation*}
    \dot x(t)=-\dfrac{k}{t}x(t)+(Tx)(t),\quad t\in[0,1],
\end{equation*}
and the Cauchy problem
\begin{equation}\label{ke-27}
\left\{
\begin{array}{l}
    \dot x(t)=-\dfrac{k}{t}x(t)+(Tx)(t),\quad t\in[0,1],\\
    x(0)=0,
\end{array}
\right.
\end{equation}
have only the trivial solution in the space $\spaceD_+$.

\section{The solvability conditions for linear boundary value problems}\label{ss-4}

Here we suppose that a linear bounded operator $T:\spaceC\to\spaceL$ is regular, that is it has the representation as a difference of positive (in a sense of conus of nonnegative functions) operators:
\begin{equation*}
    T=T^+-T^-,
\end{equation*}
where linear operators $T^+$, $T^-:\spaceC\to\spaceL$ are positive (map nonnegative functions into almost everywhere nonnegative functions). It is clear, that the norms of such operators are defined by the equality
\begin{equation*}
    \norm{T^{+/-}}_{\spaceC\to\spaceL}=\int_0^1 (T^{+/-}\unit)(t)\,dt,
\end{equation*}
where $\unit$ is the unit function: $\unit(t)=1$ for all $t\in[0,1]$.
Note, every regular operator $T:\spaceC\to\spaceL$ is $u$-bounded, that is there exists the function $u\in\spaceL$ such that for every  $x\in\spaceC$ the inequality
\begin{equation*}
    |(Tx)(t)|\le u(t)\norm{x}_{\spaceC}\text{ for almost all $t\in[0,1]$}
\end{equation*}
holds \cite[Theorem 1.27, p.~234]{KVP}.

An operator $T:\spaceC\to\spaceL$ is called a Volterra operator if for every $t^*\in(0,1)$ and  every $x\in\spaceC$ such that  $x(t)=0$ for all $t\in[0,t^*]$, the equality $(Tx)(t)=0$ holds for almost all $t\in[0,t^*]$.

If a regular operator $T$ is a Volterra operator, it is easy to show that problem \eqref{ke-27}, which is equivalent to the equation
\begin{equation}\label{ke-31}
    x(t)=\int_0^t\left(\frac{s}{t}\right)^k(Tx)(s)\,ds,\ t\in(0,1],
\end{equation}
has only the trivial solution. Therefore, for  $k>0$ the Cauchy problem \eqref{ke-25} is uniquely solvable in the space  $\spaceD_+$.

An operator $T:\spaceC\to\spaceL$ is called a Volterra operator with respect to the point $t=1$ if for every $t^*\in(0,1)$ and every $x\in\spaceC$ such that $x(t)=0$ for all $t\in[t^*,1]$, the equality $(Tx)(t)=0$ holds for almost all $t\in[t^*,1]$.

If a regular operator $T$ is a Volterra operator
with respect to the point $t=1$, then
it is easy to show that for $\ell x=x(1)$  the homogeneous  problem \eqref{ke-23}, which is equivalent to the equation
\begin{equation}\label{ke-32}
    x(t)=-\int_t^1\left(\frac{t}{s}\right)^{|k|}(Tx)(s)\,ds,\quad t\in[0,1],
\end{equation}
has only the trivial solution. Therefore, for  the Cauchy problem \eqref{ke-20} is uniquely solvable in the space  $\spaceD_-$.

Define the operators $\Lambda_-$, $\Lambda_+:\spaceL\to\spaceC$ by the equalities
\begin{equation*}
(\Lambda_-f)(t)=-\int_t^1\left(\frac{t}{s}\right)^{|k|}f(s)\,ds,\quad t\in[0,1],
\end{equation*}
\begin{equation*}
(\Lambda_+f)(t)=\int_0^t\left(\frac{s}{t}\right)^k f(s)\,ds,\ t\in(0,1].
\end{equation*}

For any linear bounded operators $T:\spaceC\to\spaceL$ the condition on the spectral radius 
\begin{equation*}
\rho(\Lambda_-T)<1
\end{equation*}
guarantees the unique solvability of problem \eqref{ke-20}, and the condition
\begin{equation*}
\rho(\Lambda_+T)<1
\end{equation*}
guarantees the unique solvability of problem \eqref{ke-25}. In particular, if
\begin{equation}\label{ke-36}
    \norm{T}_{\spaceC\to\spaceL}\le1,
\end{equation}
both problems \eqref{ke-20}, \eqref{ke-25} are uniquely solvable. In this case, the norms of $\Lambda_+T$ and $\Lambda_-T$ are less than unit, therefore, equations \eqref{ke-31}, \eqref{ke-32} have only the trivial solution.

It turns out, if we know the norms of positive and negative parts $T^+$, $T^-$ of a regular operator $T$, then the solvability condition \eqref{ke-36} can be weakened.

%\begin{theorem}\label{ket-1} Let non-negative numbers $\calT^+$, $\calT^-$ be given.
%The following statements are equivalent:
%(i) problem \eqref{ke-25} in the space  $\spaceD_+$ for $k>0$  is uniquely solvable for all %operators $T=T^+-T^-$ such that linear positive operators $T^+$, $T^-:\spaceC\to\spaceL$ %satisfy
%\begin{equation}\label{ke-37}
%        \norm{T^+}_{\spaceC\to\spaceL}=\calT^+,\quad
%        \norm{T^-}_{\spaceC\to\spaceL}=\calT^-,
%\end{equation}
%(2) the inequalities
%\begin{equation}\label{ke-38}
% \calT^+\le1, \quad   \calT^-\le 2\sqrt{1-\calT^+}
%\end{equation}
%hold.
%\end{theorem}

\begin{theorem}\label{ket-1} Let non-negative numbers $\calT^+$, $\calT^-$ be given,  $k>0$.
The inequalities
\begin{equation}\label{ke-38}
 \calT^+\le1, \quad   \calT^-\le 2\sqrt{1-\calT^+},
\end{equation}
are necessary and sufficient for
problem \eqref{ke-25} in the space  $\spaceD_+$ to be uniquely solvable for all operators $T=T^+-T^-$ such that linear positive operators $T^+$, $T^-:\spaceC\to\spaceL$ satisfy the equalities
\begin{equation}\label{ke-37}
        \norm{T^+}_{\spaceC\to\spaceL}=\calT^+,\quad
        \norm{T^-}_{\spaceC\to\spaceL}=\calT^-.
\end{equation}
\end{theorem}

For proving Theorem \ref{ket-1}  we need the following lemma, which can be proved as the analogous assertions from \cite{bravyi2}, \cite{bravyi3}.

\begin{lemma}\label{kel-1} Let non-negative numbers $\calT^+$, $\calT^-$ be given, $k>0$.  Problem \eqref{ke-25}  in $\spaceD_+$ is uniquely solvable for all operators $T=T^+-T^-$ such that linear positive operators $T^+$, $T^-:\spaceC\to\spaceL$ satisfy \eqref{ke-37}
if and only if the problem
\begin{equation}\label{ke-39}
\left\{
\begin{array}{l}
    \dot x(t)=-\dfrac{k}{t}x(t)+p_1(t)x(t_1)+p_2(t)x(t_2),\quad t\in[0,1],\\
    x(0)=0,
\end{array}
\right.
\end{equation}
has only the trivial solution for all $0\le t_1< t_2\le1$ and all $p_1$, $p_2\in\spaceL$ such that there exist
\begin{equation}\label{ke-41}
    p^+,p^-\in\spaceL,\quad p^+\ge0,\quad p^-\ge0, \quad \norm{p^{+}}=\calT^{+}, \quad \norm{p^{-}}=\calT^{-},
\end{equation}
and the conditions
\begin{equation}\label{ke-40}
    p_1+p_2=p^+-p^-,\quad     -p^-\le p_i\le p^+,\quad i=1,2,
\end{equation}
hold.
\end{lemma}

\begin{proof}[Proof of Theorem \ref{ket-1}]
Solve problem \eqref{ke-39}, which is equivalent to
\begin{equation*}
    x(t)=\int_0^t\left(\frac{s}{t}\right)^k(p_1(s)x(t_1)+p_2(s)x(t_2))\,ds,\quad t\in(0,1],
\end{equation*}
in the space $\spaceC$. This equation has only the trivial solution if and only if
\begin{equation*}
    \Delta=\left|
           \begin{array}{cc}\dd
           1-\int_0^{t_1}\left(\dfrac{s}{{t_1}}\right)^kp_1(s)\,ds & \dd           -\int_0^{t_1}\left(\dfrac{s}{{t_1}}\right)^kp_2(s)\,ds\\
\dd           -\int_0^{t_2}\left(\dfrac{s}{{t_2}}\right)^kp_1(s)\,ds & \dd 1-\int_0^{t_2}\left(\dfrac{s}{{t_2}}\right)^kp_2(s)\,ds
           \end{array}
           \right| \ne0.
\end{equation*}
For all $p_1$, $p_2$ satisfying \eqref{ke-40}, \eqref{ke-41}, we have
\begin{equation*}
    \Delta=\left|
           \begin{array}{cc}\dd
           1-\int_0^{t_1}\left(\dfrac{s}{{t_1}}\right)^kp_1(s)\,ds & \dd           1-\int_0^{t_1}\left(\dfrac{s}{{t_1}}\right)^k(p^+(s)-p^-(s))\,ds\\
\dd           -\int_0^{t_2}\left(\dfrac{s}{{t_2}}\right)^kp_1(s)\,ds & \dd 1-\int_0^{t_2}\left(\dfrac{s}{{t_2}}\right)^k(p^+(s)-p^-(s))\,ds
           \end{array}
           \right| .
\end{equation*}
Our aim is to determine for which $\calT^+$, $\calT^-$ the inequality $\Delta>0$ is fulfilled for all $p_1$, $p^+$, $p^-$ such that \eqref{ke-40}, \eqref{ke-41}, and for all
$0\le t_1< t_2\le1$.

For $p_1=0$, we have
\begin{equation*}
    \Delta=\dd     1-\int_0^{t_2}\left(\dfrac{s}{{t_2}}\right)^k(p^+(s)-p^-(s))\,ds.
\end{equation*}
All values of $\Delta\ne0$ are positive for fixed $p^+$, $p^-$
and for sufficient small $t_2>0$. Therefore, for the inequality $\Delta\ne0$ to be fulfilled for all permissible parameters, it is necessary that $\Delta>0$ for all permissible parameters.

For fixed $t_2$, $\Delta$ is minimal if the function $p^+\in\spaceL$ is ``concentrated'' at the point $s=t_2$, and the function $p^-\in\spaceL$ is ``concentrated'' at the point $s=0$. Therefore,
the inequality
\begin{equation}\label{ke-46}
    \norm{p^+}=\calT^+\le1
\end{equation}
is necessary for $\Delta>0$ provided \eqref{ke-40}, \eqref{ke-41}.

We have
\begin{equation*}
\begin{split}
    \Delta=\alpha-\alpha\int_0^{t_1}\left(\dfrac{s}{{t_1}}\right)^kp_1(s)\,ds+
    \beta\int_0^{t_2}\left(\dfrac{s}{{t_2}}\right)^kp_1(s)\,ds=\\
    =\alpha+\int_0^{t_1}\left(-\alpha\left(\dfrac{s}{{t_1}}\right)^k+
    \beta\left(\dfrac{s}{{t_1}}\right)^k\right)p_1(s)\,ds
    +\beta\int_{t_1}^{t_2}\left(\dfrac{s}{{t_2}}\right)^k p_1(s)\,ds,
\end{split}
\end{equation*}
where
%\begin{equation*}
    $\alpha= 1-\int_0^{t_2}\left(\dfrac{s}{{t_2}}\right)^k(p^+(s)-p^-(s))\,ds>0$,
    $\beta= 1-\int_0^{t_1}\left(\dfrac{s}{{t_1}}\right)^k(p^+(s)-p^-(s))\,ds>0$,
%\end{equation*}
since inequality \eqref{ke-46} holds. So, for fixed $t_1<t_2$, $p^+$, $p^-$, the functional $\Delta$ takes its minimum if
$$p_1(t)=-p^-(t),\quad t\in[t_1,t_2],$$
and
$$p_1(t)=p^+(t),\quad t\in[0,t_1),\ \text{ or }\ 
p_1(t)=-p^-(t),\quad t\in[0,t_1).$$
If
$$p_1(t)=-p^-(t),\quad t\in[0,t_2],$$
it is easy to see that $\Delta>0$ for any other parameters.
Consider the rest case
\begin{equation*}
    p_1(t)=
    \left\{
    \begin{array}{ll}
    p^+(t),&\quad t\in[0,t_1],\\
    -p^-(t),&\quad t\in(t_1,t_2].
\end{array}
\right.
\end{equation*}
Then
\begin{equation*}
    \Delta=\left|
           \begin{array}{cc}\dd
           1-\int_0^{t_1}\left(\dfrac{s}{{t_1}}\right)^kp^+(s)\,ds & \dd           \int_0^{t_1}\left(\dfrac{s}{{t_1}}\right)^kp^-(s)\,ds\\[30pt]
\begin{aligned}
\dd           -\int_0^{t_1}\left(\dfrac{s}{{t_2}}\right)^kp^+(s)\,ds+ \\
+             \int_{t_1}^{t_2}\left(\dfrac{s}{{t_2}}\right)^kp^-(s)\,ds
 \end{aligned}
&
  \begin{aligned}
   \dd 1-\int_{t_1}^{t_2}\left(\dfrac{s}{{t_2}}\right)^kp^+(s)\,ds+\\
   +\int_0^{t_1}\left(\dfrac{s}{{t_2}}\right)^kp^-(s)\,ds
  \end{aligned}
 \end{array}
\right| .
\end{equation*}
Now $\Delta$ takes its minimal value if  the functions $p^+$, $p^-$ are ``concentrated'' at $t=t_2$ on the interval $[t_1,t_2]$, and at $t=0$ or $t=t_1$ on the interval $[0,t_1)$.

Using notation
\begin{equation*}
\begin{split}
\calT^+_1=\int_0^{t_1}p^+(s)\,ds,\quad \calT^-_1=\int_0^{t_1}p^-(s)\,ds,\\
\calT^+_2=\int_{t_1}^{t_2}p^+(s)\,ds,\quad \calT^-_2=\int_{t_1}^{t_2}p^-(s)\,ds,
\end{split}
\end{equation*}
we get
\begin{equation}\label{ke-52}
    \Delta>
    \left|
           \begin{array}{cc}1-K^+ & K^- \\           -\left(\dfrac{t_1}{{t_2}}\right)^kK^++\calT^-_2&\left(\dfrac{t_1}{{t_2}}\right)^kK^-+1-\calT^+_2 \end{array}
\right|\equiv \Delta_1,
\end{equation}
where $K^+=0$ or $K^+=\calT^+_1$, $K^-=0$ or $K^-=\calT^-_1$. It follows that $\Delta$ can be arbitrarily closely to the right hand side of inequality \eqref{ke-52}. It is easy to see that $\Delta_1$ takes its minimal value at $t_1=0$.
In this case
\begin{equation*}
\Delta_1=(1-K^+)(1-\calT_2^+)-K^-\,\calT_2^-,
\end{equation*}
and $\Delta_1$ takes its minimal value, in particular, for
 $$K^+=\calT^+_1=0,\quad \calT^+_1=\calT^+,\quad K^-=\calT^-_1=\calT^-/2,\quad \calT^-_2=\calT^-/2.$$
Then
$$\Delta_1=1-\calT^+-(\calT^-)^2/4,$$
hence, $\Delta>0$ for all admissible parameters if and only if inequalities \eqref{ke-38} hold.
\end{proof}

Now consider problem \eqref{ke-20} for $k<0$.
\begin{theorem}\label{ket-3} Let non-negative numbers $\calT^+$, $\calT^-$ and $k<0$ be given.
For problem \eqref{ke-20} in the space $\spaceD_-$ to be uniquely solvable for all operators  $T=T^+-T^-$ such that linear positive operators $T^+$, $T^-:\spaceC\to\spaceL$ satisfy the equalities
\begin{equation}\label{ke-137}
        \norm{T^+}_{\spaceC\to\spaceL}=\calT^+,\quad
        \norm{T^-}_{\spaceC\to\spaceL}=\calT^-,
\end{equation}
it is necessary and sufficient that
\begin{equation}\label{ke-138}
 \calT^-\le1, \quad   \calT^+\le 2\,\sqrt{1-\calT^-}.
\end{equation}
\end{theorem}

For proving Theorem \ref{ket-3}, we need an analog of Lemma \ref{kel-1} (it can be proved similarly to appropriate statements from \cite{bravyi2}, \cite{bravyi3}).

\begin{lemma}\label{kel-2} Let non-negative numbers $\calT^+$, $\calT^-$ and
$k<0$ be given. For problem \eqref{ke-20} in the space $\spaceD_-$ to be uniquely solvablefor all operators  $T=T^+-T^-$ such that linear positive operators  $T^+$, $T^-:\spaceC\to\spaceL$ satisfy  \eqref{ke-137},
it is necessary and sufficient that the problem
\begin{equation}\label{ke-139}
\left\{
\begin{array}{l}
    \dot x(t)=-\dfrac{k}{t}x(t)+p_1(t)\,x(t_1)+p_2(t)\,x(t_2),\quad t\in[0,1],\\
    x(1)=0,
\end{array}
\right.
\end{equation}
has in the space $\spaceD_-$ only the trivial solution for all $0\le t_1< t_2\le1$ and all functions $p_1$, $p_2\in\spaceL$ such that
\begin{equation}\label{ke-141}
    p^+,p^-\in\spaceL,\quad p^+\ge0,\quad p^-\ge0, \quad \norm{p^{+}}=\calT^{+}, \quad \norm{p^{-}}=\calT^{-},
\end{equation}
and
\begin{equation}\label{ke-140}
\begin{split}
    p_1+p_2=p^+-p^-,\quad     -p^-\le p_i\le p^+,\quad i=1,2.
\end{split}
\end{equation}
\end{lemma}

\begin{proof}[Proof of Theorem \ref{ket-3}]
Problem \eqref{ke-139} is equivalent to the equation
\begin{equation*}
    x(t)=-\int_t^1\left(\frac{t}{s}\right)^{|k|}(p_1(s)x(t_1)+p_2(s)x(t_2))\,ds,\quad t\in(0,1],
\end{equation*}
in the space $\spaceC$. This equation has only the trivial solution if and only if
\begin{equation}\label{ke-143}
    \Delta=\left|
           \begin{array}{cc}\dd
           1+\int_{t_1}^{1}\left(\dfrac{{t_1}}{s}\right)^{|k|}p_1(s)\,ds & \dd           \int_{t_1}^{1}\left(\dfrac{{t_1}}{s}\right)^{|k|}p_2(s)\,ds\\
\dd           \int_{t_2}^{1}\left(\dfrac{{t_2}}{s}\right)^{|k|}p_1(s)\,ds & \dd 1+\int_{t_2}^{1}\left(\dfrac{{t_2}}{s}\right)^{|k|}p_2(s)\,ds
           \end{array}
           \right| \ne0.
\end{equation}
For all $p_1$, $p_2$ satisfying  \eqref{ke-141}, \eqref{ke-140}, we have
\begin{equation*}
    \Delta=\left|
           \begin{array}{cc}\dd
           1+\int_{t_1}^{1}\left(\dfrac{{t_1}}{s}\right)^{|k|}p_1(s)\,ds & \dd           1+\int_{t_1}^{1}\left(\dfrac{{t_1}}{s}\right)^{|k|}(p^+(s)-p^-(s))\,ds\\
\dd           \int_{t_2}^{1}\left(\dfrac{{t_2}}{s}\right)^{|k|}p_1(s)\,ds & \dd 1+\int_{t_2}^{1}\left(\dfrac{{t_2}}{s}\right)^{|k|}(p^+(s)-p^-(s))\,ds
           \end{array}
           \right| .
\end{equation*}
Now we have to determine for which $\calT^+$, $\calT^-$ the inequality $\Delta>0$ is fulfilled for all functions $p_1$, $p^+$, $p^-$ satisfying \eqref{ke-141}, \eqref{ke-140} and for all $0\le t_1< t_2\le1$.

If $p_1=0$,
\begin{equation*}
    \Delta=\dd     1+\int_{t_2}^{1}\left(\dfrac{{t_2}}{s}\right)^{|k|}(p^+(s)-p^-(s))\,ds.
\end{equation*}
For fixed $p^+$ and $p^-$ for for sufficient small $1-t_2>0$  all values  $\Delta$ are positive.
Therefore, the condition $\Delta\ne0$ in \eqref{ke-143} can be substituted by $\Delta>0$.

For a fixed  $t_2$, the value of $\Delta$ is minimal if the function $p^+\in\spaceL$ is ``concentrated'' at, and the function $p^-\in\spaceL$ at $s=t_2$. Therefore, for the inequality $\Delta>0$ to follow from \eqref{ke-141} and \eqref{ke-140}, it is necessary that
\begin{equation}\label{ke-146}
    \norm{p^-}=\calT^-\le1.
\end{equation}

We have
\begin{equation*}
\begin{split}
    \Delta=\alpha+\alpha\int_{t_1}^{1}\left(\dfrac{{t_1}}{s}\right)^{|k|}p_1(s)\,ds-
    \beta\int_{t_2}^1\left(\dfrac{{t_2}}{s}\right)^{|k|}p_1(s)\,ds=\\
    =\alpha+\int_{t_2}^{1}\left(\alpha\left(\dfrac{{t_1}}{s}\right)^{|k|}-
    \beta\left(\dfrac{{t_2}}{s}\right)^{|k|}\right)p_1(s)\,ds
    +\alpha\int_{t_1}^{t_2}\left(\dfrac{{t_2}}{s}\right)^{|k|} p_1(s)\,ds,
\end{split}
\end{equation*}
and
\begin{equation*}
\begin{split}
    \alpha\equiv 1+\int_{t_2}^{1}\left(\dfrac{{t_2}}{s}\right)^{|k|}(p^+(s)-p^-(s))\,ds>0,\\
    \beta\equiv 1+\int_{t_1}^{1}\left(\dfrac{{t_1}}{s}\right)^{|k|}(p^+(s)-p^-(s))\,ds>0,
\end{split}
\end{equation*}
since \eqref{ke-146}. Hence, for fixed $t_1<t_2$, $p^+$, $p^-$, the functional $\Delta$ takes its minimal value if
$$p_1(t)=-p^-(t),\quad t\in[t_1,t_2],$$
and
$$p_1(t)=p^+(t),\quad t\in[t_2,1]\ \text{ or }
p_1(t)=-p^-(t),\quad t\in[t_2,1].$$
If
$$p_1(t)=-p^-(t),\quad t\in[t_1,1],$$
then it is easy to see that $\Delta>0$  for any other parameters.
Consider the rest case:
\begin{equation*}
    p_1(t)=
    \left\{
    \begin{array}{ll}
    -p^-(t),&\quad t\in[t_1,t_2],\\
     p^+(t),&\quad t\in(t_2,1].
\end{array}
\right.
\end{equation*}
Then
\begin{equation*}
    \Delta=\left|
           \begin{array}{cc}
\begin{aligned}
\dd           1+\int_{t_2}^{1}\left(\dfrac{{t_1}}{s}\right)^{|k|}p^+(s)\,ds- \\
-             \int_{t_1}^{t_2}\left(\dfrac{{t_1}}{s}\right)^{|k|}p^-(s)\,ds
 \end{aligned}
&
  \begin{aligned}
   \dd \int_{t_1}^{t_2}\left(\dfrac{{t_1}}{s}\right)^{|k|}p^+(s)\,ds-\\
   -\int_{t_2}^{1}\left(\dfrac{{t_1}}{s}\right)^{|k|}p^-(s)\,ds
  \end{aligned}\\[40pt]
\dd
           \int_{t_2}^{1}\left(\dfrac{{t_2}}{s}\right)^{|k|}p^+(s)\,ds & \dd           1-\int_{t_2}^{1}\left(\dfrac{{t_2}}{s}\right)^{|k|}p^-(s)\,ds
 \end{array}
\right|.
\end{equation*}
Now $\Delta$ takes its minimum if the functions
$p^+$ and $p^-$ are ``concentrated'' at $t=t_1$ on the interval $[t_1,t_2]$, and at the points  $t=t_2$ or
$t=1$ on the interval $(t_2,1]$.

Using the notation
\begin{equation*}
\begin{split}
\calT^+_1=\int_{t_1}^{t_2}p^+(s)\,ds,\quad \calT^-_1=\int_{t_1}^{t_2}p^-(s)\,ds,\\
\calT^+_2=\int_{t_2}^{1}p^+(s)\,ds,\quad \calT^-_2=\int_{t_2}^{1}p^-(s)\,ds,
\end{split}
\end{equation*}
and taking into account that $\Delta$ can not take the specified minimum for integrable functions $p^+$ and $p^-$, we get that
\begin{equation}\label{ke-152}
    \Delta>
    \left|
           \begin{array}{cc}1-\calT^-_1+\left(\dfrac{t_1}{{t_2}}\right)^{|k|}\calT^+_2 K^+ & \calT^+_1-\left(\dfrac{t_1}{{t_2}}\right)^{|k|}\calT^-_2 K^- \\[20pt]           \calT^+_2 K^+&1-\calT^-_2 K^- \end{array}
\right|\equiv \Delta_1.
\end{equation}
where $K^+=1$ or $K^+=t_2^{|k|}$, $K^-=1$ or $K^-=t_2^{|k|}$, and $\Delta$ can be arbitrary closely to the right-hand side of \eqref{ke-152}. It can easily be checked that for fixed rest parameters, $\Delta_1$ takes its minimal value at  $t_1=0$.
In this case
\begin{equation*}
\Delta_1=(1-\calT_2^-K^-)(1-\calT_1^-)-\calT_1^+\,\calT_2^+\,K^+,
\end{equation*}
and $\Delta_1$ takes its minimal value, in particular, for
 $$K^+=1,\quad\calT^-_1=0,\quad \calT^-_2=\calT^-,\quad K^+=1,\quad \calT^+_1=\calT^+_2=\calT^+/2.$$
Then
$$\Delta_1=1-\calT^--(\calT^+)^2/4,$$
hence, $\Delta>0$ for all admissible parameters if and only if inequalities \eqref{ke-138} hold.
\end{proof}

For comparison, we give a result on the solvability the Cauchy problem for a non-singular equation.
\begin{theorem}[\cite{Hakl}]\label{ket-0} Let non-negative numbers $\calT^+$, $\calT^-$ be given. For the Cauchy problem in the space $\spaceAC$
\begin{equation*}
\left\{
\begin{array}{l}
    \dot x(t)=(Tx)(t)+f(t),\quad t\in[0,1],\\
    x(0)=c,
\end{array}
\right.
\end{equation*}
to be uniquely solvable for all operators $T=T^+-T^-$ such that linear positive operators  $T^+$, $T^-:\spaceC\to\spaceL$ satisfy the equalities
\begin{equation*}
        \norm{T^+}_{\spaceC\to\spaceL}=\calT^+,\quad
        \norm{T^-}_{\spaceC\to\spaceL}=\calT^-,
\end{equation*}
it is necessary and sufficient that
\begin{equation*}
 \calT^+<1, \quad   \calT^-<1+ 2\sqrt{1-\calT^+}.
\end{equation*}
\end{theorem}

\section{A more general ``model'' equation}\label{p-4}\label{ss-5}

The assertions on the solvability of singular equations in the space $\spaceD_+$ and boundary value problems in the space $\spaceD_-$ admit the following natural generalizations.

Let $p:(0,1]\to\spaceR$ be a positive function and its restriction on each interval  $[\varepsilon,1]$ integrable: $p\in\spaceL[\varepsilon,1]$ for every $\varepsilon\in(0,1)$. Suppose
\begin{equation}\label{ke-54}
\lim\limits_{\varepsilon\to0+}\int_\varepsilon^1 p(t)\,dt=\infty.
\end{equation}
For example, $$p(t)=\dfrac{1}{t^\mu},\quad t\in(0,1],\quad \text{ where }\mu\ge1.$$

As a ``model'' equation, we take the ordinary differential equation
\begin{equation}\label{ke-55}
    \dot x(t)=-k\,p(t)x(t)+f(t),\quad t\in(0,1],\quad \text{ where }k\ne0.
\end{equation}

As previously for \eqref{ke-1},
a locally absolutely continuous in $(0,1]$ function $x:(0,1]\to\spaceR$ (that is $x\in\spaceAC[\varepsilon,1]$ for each $\varepsilon\in(0,1)$)
is called a solution of equation \eqref{ke-55} if $x$ satisfies \eqref{ke-55} almost everywhere. Since equation \eqref{ke-55} on each interval $[\varepsilon,1]$ ($\varepsilon\in(0,1)$) does not have singularities, then any its solution has a representation
\begin{equation}\label{ke-56}
    x(t)=x_1(t)x(1)-\int_t^1 \frac{x_1(t)}{x_1(s)}f(s)\,ds=
    x_1(t)\bigl(x(1)-\int_t^1\frac{f(s)}{x_1(s)}\,ds\bigr),\quad t\in(0,1],
\end{equation}
where
\begin{equation}\label{ke-57}
    x_1(t)=e^{k\int_t^1p(s)\,ds},\quad t\in(0,1].
\end{equation}
Obviously, that for $k>0$ the function $x_1$ decreases on  $(0,1]$ and $\lim\limits_{t\to0+}x_1(t)=\infty$;
for $k<0$ the function $x_1$ increases on $(0,1]$ and $\lim\limits_{t\to0+}x_1(t)=0$.
Denote by $\spaceD$ the set of all solutions to \eqref{ke-55} for all $f\in\spaceL[0,1]$.

Representation \eqref{ke-56} sets a one-to-one correspondence between $\spaceD$ and $\spaceL\times\spaceR$, therefore, $\spaceD$ is a Banach space with respect to the norm
\begin{equation}\label{ke-300}
    \norm{x}_{\spaceD}=|x(1)|+\int_0^1|\dot x(s)+k\,p(s)\,x(s)|\,ds.
\end{equation}

Let $k<0$. Denote $\spaceD_-=\spaceD$.
It is easy to prove that the embedding $x\to x$  from $\spaceD_-$ into $\spaceAC$ is bounded.

Indeed, let $x\in\spaceD_-$ be a solution to \eqref{ke-55}. Then
\begin{equation*}
    \norm{x}_{\spaceD_-}=|x(1)|+\int_0^1|f(s)|\,ds,
\end{equation*}
\begin{equation*}
\begin{split}
    \norm{x}_{\spaceAC}=|x(1)|+\int_0^1|\dot x(s)|\,ds=|x(1)|+\int_0^1|f(s)-k\,p(s)x(s)|\,ds\le \\
    |x(1)|+\int_0^1|f(s)|\,ds+\int_0^1|k|\,p(s)|x(s)|\,ds\le |x(1)|+\\
    \int_0^1|f(s)|\,ds+\int_0^1|k|\,p(s)\,x_1(s)\,ds\,|x(1)|+|k|\int_0^1 p(t)x_1(t)\int_t^1 \frac{|f(s)|}{x_1(s)}\,ds\,dt= \\
    |x(1)|+    \int_0^1|f(s)|\,ds+|x(1)|+|k|\int_0^1 p(t)x_1(t)\int_t^1 \frac{|f(s)|}{x_1(s)}\,ds\,dt.
\end{split}
\end{equation*}
Changing the order of integration in the last integral (it is possible by the Fubini theorem, since all integrands are non-negative), we get:
\begin{equation*}
\begin{split}
|k|\int_0^1 p(t)x_1(t)\int_t^1 \frac{|f(s)|}{x_1(s)}\,ds\,dt=\\
=\int_0^1 \int_0^s |k| p(t)x_1(t)\, dt \frac{|f(s)|}{x_1(s)}\,ds=\int_0^1(x_1(s)-x_1(0))\frac{|f(s)|}{x_1(s)}\,ds=\int_0^1|f(s)|\,ds.
\end{split}
\end{equation*}
Therefore,
\begin{equation}\label{ke-60000}
\begin{split}
    \norm{x}_{\spaceAC}\le 2|x(1)|+2\int_0^1|f(s)|\,ds\le 2 \norm{x}_{\spaceD_-},
\end{split}
\end{equation}
and for every $x\in\spaceD_-$ we have $\dot x\in\spaceL$. So, there exists a finite limit $x(0+)$. Extend elements of $\spaceD_-$ at $t=0$ continuously. For every $x\in\spaceD_-$ there exists $f\in\spaceL$ such that
$${k}p(t)x(t)=f(t)-\dot x(t),\quad t\in[0,1],$$
where the right-hand side is integrable on $[0,1]$. So,
$$\int_0^1 p(s)|x(t)|\,dt< +\infty,$$
which implies for the continuous function $x$ that $x(0)=0$. Therefore,
$$\lim\limits_{t\to0} x(t)=0\text{ for all }x\in\spaceD_-.$$
Inequality \eqref{ke-60000} means that the space $\spaceD_-$ is continuously embedded into $\spaceAC$.

For $k>0$, it follows from \eqref{ke-56} that a solution of \eqref{ke-55} can have a finite limit at $t=0+$ if and only if
\begin{equation}\label{ke-8-new}
    x(1)=\int_0^1\frac{f(s)}{x_1(s)}\,ds.
\end{equation}
In this case a solution has the representation
\begin{equation}\label{ke-58}
    x(t)=\int_0^t\frac{x_1(t)}{x_1(s)}f(s)\,ds,\ t\in(0,1],
\end{equation}
and, obviously, has the zero limit at  $t=0+$. Indeed, it follows from \eqref{ke-58} that
$$|x(t)|\le \int_0^t\frac{x_1(t)}{x_1(s)}|f(s)|\,ds\le \int_0^t|f(s)|\,ds,\quad t\in(0,1],$$
hence $\lim\limits_{t\to0+} x(t) =0$ by the absolutely continuousness of the Lebesgue integral.

For $k>0$, denote by $\spaceD_+$ the space of all solutions of \eqref{ke-55} satisfying condition \eqref{ke-8-new} and extended by zero at $t=0$. It is a Banach space equipped with the norm
\begin{equation}\label{ke-59}
    \norm{x}_{\spaceD_+}=|x(1)|+\int_0^1|\dot x(s)+{k}\,p(s)x(s)|\,ds.
\end{equation}
The space $\spaceD_+$ and  $\spaceL$ are isomorphic. A one-to-one correspondence between $\spaceD_+$ and $\spaceL$ is set by equality \eqref{ke-58}.
It is easy to show that $\spaceD_+$ is embedded into $\spaceAC$ continuously.

Each of the sets $\spaceD_-$, $\spaceD_+$ is a proper subset of all functions from $\spaceAC$
satisfying the condition $x(0)=0$. The sets $\spaceD_-$ and $\spaceD_+$ do not depend on the number $k\ne0$ and moreover, they coincide:
$$\{x|x\in\spaceD_-\}=\{x|x\in\spaceD_+\}.$$
We show that for any $k\ne0$ such a set coincides with the set $\w\spaceAC_0$ of all absolutely continuous functions $x:[0,1]\to\spaceR$ such that $x(0)=0$ and the function $p(t)\,x(t)$, $t\in(0,1]$, is integrable. If $x\in\spaceD_-$ or $x\in\spaceD_+$ for some $k\ne0$, then $x(0)=0$, $x\in\spaceAC$, and for some $f\in\spaceL$ the equality
$$p(t)\,x(t)=\frac{1}{k}\left(f(t)-\dot x(t) \right), \quad t\in(0,1],$$
holds. This implies that $x\in\w\spaceAC_0$. Conversely, if $x\in\w\spaceAC_0$, then for any  $k\ne0$ the function
$$f(t)=\dot x(t)+k\,p(t)\,x(t), \quad t\in(0,1],$$
is integrable, therefore, $x$ belongs to every sets $\spaceD_-$ and $\spaceD_+$.

For every positive $k$ the space $\spaceD_+$ is isomorphic to the space $\spaceL$, therefore
all norms \eqref{ke-59} defined on the set  $\w\spaceAC_0$ for different $k>0$ are equivalent. Similarly for every negative $k$ the space $\spaceD_-$ is isomorphic to the space  $\spaceL\times\spaceR$.
So, all norms \eqref{ke-300} defined on the set  $\w\spaceAC_0$ for different $k<0$ are equivalent. Since there are no linear bounded invertible maps between $\spaceL\times\spaceR$ and $\spaceL$, then norms in the spaces $\spaceD_+$ and $\spaceD_-$ are not equivalent.

From \eqref{ke-56} it follows that for $k<0$ in the space $\spaceD_-$ the boundary value problem
\begin{equation}\label{ke-315}
\left\{
\begin{array}{l}
    \dot x(t)=-k\,p(t)\,x(t)+f(t),\quad t\in[0,1],\\
    x(1)=c,
\end{array}
\right.
\end{equation}
is uniquely solvable (that is it has a unique solution $x\in\spaceD_-$, which is defined by \eqref{ke-56}, for all  $f\in\spaceL$, $c\in\spaceR$). Clearly, that for $k=0$  problem \eqref{ke-315} is uniquely solvable in the space $\spaceAC$, but not in the space $\spaceD_-$.

From  \eqref{ke-58} it follows that
for $k>0$  in the space $\spaceD_+$ the equation
\begin{equation}\label{ke-316}
    \dot x(t)=-k\,p(t)\,x(t)+f(t),\quad t\in[0,1],\\
\end{equation}
is uniquely solvable (it has a unique solution $x\in\spaceD_+$ defined by \eqref{ke-58} for all $f\in\spaceL$). Note, equation \eqref{ke-316} in the space  $\spaceD_+$ can be written in the equivalent form of the Cauchy problem
\begin{equation}\label{ke-317}
\left\{
\begin{array}{l}
    \dot x(t)=-k\,p(t)\,x(t)+f(t),\quad t\in[0,1],\\
    x(0)=0.
\end{array}
\right.
\end{equation}
Now and for $k=0$, problem \eqref{ke-317} (for the non-singular equation) is uniquely solvable in the space $\spaceAC$.

Let $\ell:\spaceAC\to\spaceR$ be a linear bounded functional,
$T:\spaceC[0,1]\to \spaceL[0,1]$ a linear bounded operator.
Since $\spaceD_-$ is isomorphic to  $\spaceL\times\spaceR$ and continuously embedded into  $\spaceAC$, then for $k<0$ the boundary value problem
\begin{equation}\label{ke-3319}
\left\{
\begin{array}{l}
    \dot x(t)=-k\,p(t)\,x(t)+(Tx)(t)+f(t),\quad t\in[0,1],\\
    \ell x=c,
\end{array}
\right.
\end{equation}
has the Fredholm property The proof is similar to the proof for problem \eqref{ke-19}. In particular, the problem
\begin{equation}\label{ke-3320}
\left\{
\begin{array}{l}
    \dot x(t)=-k\,p(t)\,x(t)+(Tx)(t)+f(t),\quad t\in[0,1],\\
    x(1)=c,
\end{array}
\right.
\end{equation}
has the Fredholm property.

Boundary value problems with singularities as in \eqref{ke-3319} and \eqref{ke-3320} were considered in \cite{LabShind}.

As for equation \eqref{ke-24}, using the continuity of the embedding  $\spaceD_+$ for
 $k>0$ into $\spaceAC$, we prove that the equation in the space $\spaceD_+$
\begin{equation}\label{ke-60}
    \dot x(t)=-k\,p(t)x(t)+(Tx)(t)+f(t),\quad t\in[0,1],
\end{equation}
which is equivalent to the Cauchy problem
\begin{equation}\label{ke-61}
\left\{
\begin{array}{l}
    \dot x(t)=-k\,p(t)x(t)+(Tx)(t)+f(t),\quad t\in[0,1],\\
    x(0)=0,
\end{array}
\right.
\end{equation}
has the Fredholm property for every linear bounded operator $T:\spaceC\to\spaceL$. Therefore, equation \eqref{ke-60} is uniquely solvable if and only if
the homogeneous equation
\begin{equation*}
    \dot x(t)=-k\,p(t)x(t)+(Tx)(t),\quad t\in[0,1],
\end{equation*}
or the Cauchy problem
\begin{equation}\label{ke-63}
\left\{
\begin{array}{l}
    \dot x(t)=-k\,p(t)x(t)+(Tx)(t),\quad t\in[0,1],\\
    x(0)=0,
\end{array}
\right.
\end{equation}
has only the trivial solution in the space $\spaceD_+$.

In a regular operator $T:\spaceC\to\spaceL$ is a Volterra operator, it is easy to show that problem \eqref{ke-63}, which is equivalent to the equation
\begin{equation*}
    x(t)=\int_0^t\frac{x_1(t)}{x_1(s)}(Tx)(s)\,ds,\ t\in(0,1],
\end{equation*}
has only the trivial solution, therefore problem  \eqref{ke-61} for $k>0$ is uniquely solvable in $\spaceD_+$.

If a regular operator $T:\spaceC\to\spaceL$ is a Volterra operator with respect to the point $t=1$, then for $k<0$ problem \eqref{ke-3320} in the space $\spaceD_+$ is uniquely solvable (also as well as problem \eqref{ke-20}).

Introduce the operators $\Lambda_-$, $\Lambda_+:\spaceL\to\spaceC$ by the equalities
\begin{equation*}
(\Lambda_-f)(t)=-\int_t^1\frac{x_1(t)}{x_1(s)}f(s)\,ds,\quad t\in[0,1],
\end{equation*}
\begin{equation*}
(\Lambda_+f)(t)=\int_0^t\frac{x_1(t)}{x_1(s)} f(s)\,ds,\ t\in(0,1].
\end{equation*}

For a linear bounded operator
$T:\spaceC\to\spaceL$, the condition on the spectral radius
\begin{equation*}
\rho(\Lambda_-T)<1
\end{equation*}
guarantees the unique solvability of problem \eqref{ke-3320}, and the condition
\begin{equation*}
\rho(\Lambda_+T)<1
\end{equation*}
guarantees the unique solvability of problem \eqref{ke-61}. In particular, if
\begin{equation}\label{ke-3336}
    \norm{T}_{\spaceC\to\spaceL}\le1,
\end{equation}
then problems \eqref{ke-3320}, \eqref{ke-61} are uniquely solvable. Indeed, in this case the norms of operators $\Lambda_+T$ and $\Lambda_-T$ are less than unit, so the equations in the space $\spaceC$ that are equivalent to the corresponding boundary value problems have a unique solution.

As well as for the ``model'' equation \eqref{ke-1}, if we know the norm of the positive and negative parts  $T^+$, $T^-$ of a regular operator $T$, then the conditions of the solvability for \eqref{ke-3336} can be weaken. The assertions of Theorems \ref{ket-1} and \ref{ket-3} can be transferred on the more general problems \eqref{ke-61} and \eqref{ke-3320} without changing.

\begin{theorem}\label{ket-2} Let non-negative numbers $\calT^+$, $\calT^-$ and $k>0$ be given. For problem \eqref{ke-61} in the space $\spaceD_+$ to be uniquely solvable for all operators $T=T^+-T^-$ such that linear positive operators $T^+$, $T^-:\spaceC\to\spaceL$ satisfy the equalities
\begin{equation}\label{ke-65}
        \norm{T^+}_{\spaceC\to\spaceL}=\calT^+,\quad
        \norm{T^-}_{\spaceC\to\spaceL}=\calT^-,
\end{equation}
it is necessary and sufficient to have
\begin{equation*}
 \calT^+\le1, \quad   \calT^-\le 2\,\sqrt{1-\calT^+}.
\end{equation*}
\end{theorem}

\begin{theorem}\label{ket-4} Let non-negative numbers $\calT^+$, $\calT^-$ and $k<0$ be given.
For the problem
\begin{equation*}
\left\{
\begin{array}{l}
    \dot x(t)=-k\,p(t)x(t)+(Tx)(t),\quad t\in[0,1],\\
    x(1)=c,
\end{array}
\right.
\end{equation*}
in the space  $\spaceD_-$
to be uniquely solvable for all operators $T=T^+-T^-$ such that linear positive operators $T^+$, $T^-:\spaceC\to\spaceL$ satisfy \eqref{ke-65},
it is necessary and sufficient to have
\begin{equation*}
 \calT^-\le1, \quad   \calT^+\le 2\,\sqrt{1-\calT^-}.
\end{equation*}
\end{theorem}

\section{Equations with singularities at derivatives}\label{ss-6}

In the works by  A.E.~Zernov, R.P. Agarwal, D. O'Regan,  it is shown that conditions for the unique solvability can be obtained also when non-integrable singularities are contained in terms with functional operators. Consider an equation in which non-integrable singularities are permitted in all terms of the right-hand side:
\begin{equation}\label{n-1}
    \frac{1}{p(t)}\dot x(t)=-k\,x(t)+(Tx)(t)+f(t),\quad t\in[0,1],
\end{equation}
where the function $p:(0,1]\to\spaceR$, as well as in section \ref{p-4}, is positive and its restriction on every interval $[\varepsilon,1]$ is integrable ($p\in\spaceL[\varepsilon,1]$ for every $\varepsilon\in(0,1)$), and the function $p$ satisfies condition \eqref{ke-54}. Namely equations of the form \eqref{n-1} (with a neutral term, depending on the derivative $\dot x$, and with a non-linear term) were considered in  \cite{Zernov-2001-53-4}, \cite{Zernov-1}.

Here we find conditions on a linear bounded operator  $T:\spaceC\to\spaceL$ and a function $f\in\spaceL$ guaranteeing the solvability of the Cauchy problem for an equation that is equivalent to \eqref{n-1}.
For this,  because of additional singularities, we have to constrict  the range of the operator $T$ and the set of the right-hand sides $f$ in comparison to the previous sections.

Suppose that $\nu:[0,1]\to\spaceR$ is an increasing function satisfying the condition $\nu(0)=0$,
the function $p$ is defined previously, $\spaceL_\infty^{p\nu}$ is the space of all measurable functions $z:[0,1]\to\spaceR$ such that
\begin{equation*}
\vraisup_{s\in[0,1]}\frac{|z(s)|}{p(s)\,\nu(s)}<\infty,
\end{equation*}
with the norm
\begin{equation*}
    \norm{z}_{\spaceL_\infty^{p\nu}}=     \norm{|z|/(p\nu)}_{\spaceL_\infty}.
\end{equation*}
Consider the following ``model'' equation in the weight space $\spaceL_\infty^{p\nu}$
\begin{equation}\label{rr-1}
    \dot x(t)=-k\,p(t)\,x(t)+g(t),\quad t\in[0,1],
\end{equation}
where $g=\nu\,p\,f$, $f\in\spaceLi$.

A function $x\in\spaceAC[\varepsilon,1]$ for every $\varepsilon\in(0,1)$ is called a solution of equation \eqref{rr-1} if $x$ satisfies this equation for almost all $t\in[0,1]$. Every solution of \eqref{rr-1} has the representation
\begin{equation}\label{rr-2}
x(t)=x_1(t)\left(x(1)-\int_t^1\frac{p(s)\,\nu(s)\,f(s)}{x_1(s)}\,ds\right),\quad t\in[0,1],
\end{equation}
where $f\in\spaceLi$, function $x_1$ is defined by \eqref{ke-57}:
\begin{equation*}
    x_1(t)=e^{k\int_t^1p(s)\,ds},\quad t\in(0,1].
\end{equation*}
Denote the set of such solutions by $\spaceD$. \eqref{rr-2} sets a one-to-one correspondence between $\spaceD$ and $\spaceLip\times\spaceR$ or $\spaceLi\times\spaceR$. So, $\spaceD$ is a Banach space with the norm
\begin{equation}\label{rr-nd}
    \norm{x}_{\spaceD}=|x(1)|+\vraisup_{s\in[0,1]}\frac{|\dot x(s)+kp(s)x(s)|}{p(s)\nu(s)}.
\end{equation}

Let $k<0$. Then $\lim\limits_{t\to0+}x_1(t)=0$. Moreover, for the element $x\in\spaceD$ with representation \eqref{rr-2}, we have
\begin{equation*}
    |x(t)|\le x_1(t)|x(1)| +\nu(t)\norm{f}_\spaceLi\int_t^1\frac{x_1(t)}{x_1(s)}p(s)\,ds\le
    x_1(t)|x(1)| +\nu(t)\norm{f}_\spaceLi/k,
\end{equation*}
therefore $\lim\limits_{t\to0+}x(t)=0$ and $\vraisup\limits_{t\in[0,1]}\frac{|x(t)|}{\nu(t)}<+\infty$ if
\begin{equation}\label{rr-pnu}
\vraisup\limits_{t\in[0,1]}\frac{|x_1(t)|}{\nu(t)}<+\infty.
\end{equation}

Let $\spaceD_-\equiv\spaceD_-^{k,\nu,p}$ be the Banach space of all elements from $\spaceD$ extended by zero at $t=0$ with the norm \eqref{rr-nd}. Then each element from $\spaceD_-$ has representation \eqref{rr-2}, which defines an isomorphism between $\spaceD_-$ and $\spaceLi\times\spaceR$. The space  $\spaceD_-$ is continuously embedded into $\spaceC$: from \eqref{rr-2} it follows that
\begin{equation*}
\begin{split}
    \norm{x}_{\spaceC}\le |x(1)|+\frac{\norm{\nu}_{\spaceC}}{k}  \norm{f}_{\spaceLi}\le\max\{1,\frac{\norm{\nu}_{\spaceC}}{k}\}(|x(1)|+\norm{f}_{\spaceLi})=\\
    =\max\{1,\frac{\norm{\nu}_{\spaceC}}{k}\}\norm{x}_{\spaceD}.
\end{split}
\end{equation*}
Moreover,  $\spaceD_-$ is continuously embedded into the weight space
$\spaceL_\infty^{1/\nu}$ if  condition \eqref{rr-pnu} holds:
\begin{equation*}
    \norm{x}_{\spaceL_\infty^{1/\nu}}\le
     \vraisup\frac{|x_1(t)|}{\nu(t)}|x(1)|+
     \frac{1}{k}\norm{f}_{\spaceLi}\le  \max\{\vraisup\frac{|x_1(t)|}{\nu(t)},1/k\}\norm{x}_{\spaceD}.
\end{equation*}

Let $k>0$. The function
\begin{equation*}
  K(t)=  \int_t^1\frac{p(s)\,\nu(s)\,f(s)}{x_1(s)}\,ds,\quad t\in(0,1],
\end{equation*}
is bounded. Indeed,
\begin{equation*}
  |K(t)|\le \frac{v(t)\norm{f}_{\spaceLi}}{k} \int_t^1 e^{-k\int_s^1p(\tau)\,d\tau}\,d\left(-k\int_s^1p(\tau)\,d\tau\right)=\frac{v(t)\norm{f}_{\spaceLi}}{k},\  t\in(0,1].
\end{equation*}
It is easy to prove that there exists a finite limit
$\lim\limits_{t\to0+} K(t)$.

Let $x$ be a solution of \eqref{rr-1}. Since $\lim\limits_{t\to0+} x_1(t)=\infty$, then it follows from \eqref{rr-2} that $x\in\spaceD$ has a finite limit if and only if
\begin{equation}\label{rr-3}
    x(1)=\int_0^1\frac{p(s)\,\nu(s)\,f(s)}{x_1(s)}\,ds.
\end{equation}
In this case
\begin{equation}\label{rr-4}
  x(t)=\int_0^t\frac{x_1(t)}{x_1(s)}p(s)\nu(s)f(s)\,ds,\ t\in(0,1].
\end{equation}
Then $\lim\limits_{t\to0+} x(t)=0$, since
\begin{equation}\label{rr-4-1}
  |x(t)|\le \nu(t)\norm{f}_{\spaceLi}\int_0^t\frac{x_1(t)}{x_1(s)}p(s)\,ds=\nu(t)\norm{f}_{\spaceLi}/k,\quad t\in(0,1].
\end{equation}

Let $\spaceD_+\equiv\spaceD_+^{k,\nu,p}$ be the space of all solutions of \eqref{n-1} satisfying condition \eqref{rr-3} and extended by zero at $t=0$. Every element from $\spaceD_+$ has representation \eqref{rr-4}, which defines an isomorphism between $\spaceD_+$ and $\spaceLi$. Equipped with the norm of the space  $\spaceD$, it is a Banach space isomorphic to $\spaceL$.  The space $\spaceD_+$ is continuously embedded into the space $\spaceC$: from \eqref{rr-4-1} it follows that
\begin{equation*}
    \norm{x}_{\spaceC}\le \frac{\norm{\nu}_{\spaceC}}{k}  \norm{f}_{\spaceLi}\le\max\{1,\frac{\norm{\nu}_{\spaceC}}{k}\}(|x(1)|+\norm{f}_{\spaceLi})=\max\{1,\frac{\norm{\nu}_{\spaceC}}{k}\}\norm{x}_{\spaceD}.
\end{equation*}
Moreover, the space $\spaceD_+$ is continuously embedded into the weight space
$\spaceL_\infty^{1/\nu}$:
\begin{equation*}
    \norm{x}_{\spaceL_\infty^{1/\nu}}\le \frac{1}{k}\norm{f}_{\spaceLi}\le  \max\{1,1/k\}\norm{x}_{\spaceD}.
\end{equation*}

\subsection{Problem for $k>0$}

For  $k>0$ in the space $\spaceD_+$, we consider the Cauchy problem
\begin{equation}\label{n-2}
\left\{
\begin{array}{l}
    \dot x(t)=-k\,p(t)\,x(t)+p(t)\,\nu(t)\,(Tx)(t)+p(t)\,\nu(t)\,f(t),\quad t\in[0,1],\\
    x(0)=0,
\end{array}
\right.
\end{equation}
where $f\in\spaceLi$, $T:\spaceC\to\spaceLi$ is a linear bounded operator. This problem is equivalent to the following equation in the space $\spaceC$:
\begin{equation}\label{rr-5}
  x(t)=(ATx)(t)+(Af)(t),\quad t\in(0,1],
\end{equation}
where
\begin{equation*}
  (Az)(t)=\int_0^t\frac{x_1(t)}{x_1(s)}p(s)\nu(s)z(s)\,ds,\quad t\in(0,1].
\end{equation*}

If $T$ is a Volterra operator, then \eqref{rr-5} has a unique solution in the space  $\spaceC$, therefore problem\eqref{n-2} is uniquely solvable. Indeed, equation \eqref{rr-5} with a Volterra operator $AT:\spaceC\to\spaceC$ can be considered on every interval $[0,\theta]$, $\theta\in(0,1]$. Estimate the norm of this operator:
\begin{equation*}
  \norm{AT}_{\spaceC[0,\theta]\to\spaceC[0,\theta]}\le \nu(\theta)\frac{\norm{T}_{\spaceC\to\spaceLi}}{k}.
\end{equation*}

Since $\lim\limits_{t\to0}\nu(t)=0$, then for small enough $\theta>0$ the operator $AT$ is a contraction, therefore,  equation \eqref{rr-5} has a unique solution on  $[0,\theta]$. As far as the equation has no singularities on $[\theta,1]$, then from the Volterra property and regularity of the operator  $T$  it follows that the equation has a unique solution on the whole  interval $[0,1]$.

If the operator  $T:\spaceC\to\spaceLi$ is not a Volterra one, then, for example,
the condition
\begin{equation*}%\label{rr-7}
    \norm{T}_{\spaceC\to\spaceLi}\le \frac{1}{\nu(1)\,k}
\end{equation*}
guarantees the contraction of the operator $AT$ and, therefore, the unique solvability of  \eqref{n-2}.

From  \cite{Zernov-1}, it follows that for the solvability of problem
\eqref{n-3} the term with the operator $T$, generally speaking, can not contain the improving
factor $\nu(t)$.

Suppose $T:\spaceC\to\spaceLi$ is a regular operator such that
\begin{equation}\label{rr-8}
    \vraisup_{t\in[0,1]}\frac{(|T|\nu)(t)}{\nu(t)}\equiv\calT<+\infty,
\end{equation}
where
$$T=T^+-T^-,\quad |T|=T^++T^-,$$
$T^+$, $T^-:\spaceC\to\spaceLi$ are linear positive operators.
Any Volterra operator satisfies condition \eqref{rr-8}. In this case, we have
\begin{equation*}
    \vraisup_{t\in[0,1]}\frac{(|T|\nu)(t)}{\nu(t)}\le
    \vraisup_{t\in[0,1]}\frac{\nu(t)(|T|\unit)(t)}{\nu(t)}=\norm{|T|}_{\spaceC\to\spaceLi}<+\infty.
\end{equation*}
Consider with such an operator $T$ and for $k>0$ the Cauchy problem
\begin{equation}\label{n-20}
\left\{
\begin{array}{l}
    \dot x(t)=-k\,p(t)\,x(t)+p(t)\,(Tx)(t)+p(t)\,\nu(t)\,f(t),\quad t\in[0,1],\\
    x(0)=0,
\end{array}
\right.
\end{equation}
on the set
$$\Omega_C=\{x\in\spaceD_+:\ \vraisup_{s\in[0,1]}\frac{|x(t)|}{\nu(t)}\le C\},$$
where $C$ is a positive constant.

Let $$\spaceC_{C\nu}=\{x\in\spaceC:\ \vraisup_{s\in[0,1]}\frac{|x(t)|}{\nu(t)}\le C\}.$$

It is clear that  $\spaceC_{C\nu}$ is a Banach space with the norm
\begin{equation*}
    \norm{x}_{\Omega_C}=\vraisup_{t\in[0,1]}\frac{|x(t)|}{\nu(t)}.
\end{equation*}

If $x\in\Omega_C$, we have
\begin{equation*}
    \vraisup_{t\in[0,1]}     \frac{|(Tx)(t)|}{\nu(t)}<+\infty,
\end{equation*}
so, problem \eqref{n-20} is equivalent to the equation in $\spaceC_{C\nu}$:
\begin{equation}\label{rr-9}
  x(t)=(ATx)(t)+(A(\nu f))(t),\quad t\in(0,1],
\end{equation}
where
\begin{equation*}
  (Az)(t)=\int_0^t\frac{x_1(t)}{x_1(s)}p(s)z(s)\,ds,\quad t\in(0,1].
\end{equation*}
Since
\begin{equation*}
  |A(\nu f)(t)|\le v(t)\frac{\norm{f}_{\spaceLi}}{k},\quad t\in(0,1],
\end{equation*}
and for $x\in\Omega_C$ (and for $x\in\spaceC_{C\nu}$)
\begin{equation*}%\label{rr-11}
  |(ATx)(t)|\le v(t)C\frac{\calT}{k},\quad t\in[0,1],
\end{equation*}
then for
\begin{equation}\label{rr-12}
    \calT<k
\end{equation}
the operator of the right-hand side of \eqref{rr-9} maps the set $x\in\Omega_C$ into itself for large enough $C$ and it is a contraction.

Any solution of  \eqref{n-20} in $\spaceD_+$ is bounded, therefore, it belongs to the set $\Omega_C$ for some large enough $C\in\spaceR$.

We conclude that if conditions  \eqref{rr-8}, \eqref{rr-12} hold (that is
\begin{equation}\label{rr-13}
    \calT\equiv \vraisup_{t\in[0,1]}\frac{(|T|\nu)(t)}{\nu(t)}<k),
\end{equation}
then problem \eqref{n-20} in the space $\spaceD_+$ has a unique solution.

Condition \eqref{rr-13} is essential: we cannot change this inequality with the non-strict inequality. Indeed, if $Tx=kx$, then problem \eqref{n-20}, generally speaking, has no solutions for  $p\nu\not\in\spaceL$.

In comparison to \cite{Zernov-1} for the linear case, we have changed the operator from \cite{Zernov-1}
$$(Tx)(t)=ax(bt),\quad t\in[0,1],\text{ for $a,b\in\spaceR$, $b\in(0,1]$,}$$
by a general functional regular operator $T$. Moreover, as a factor $\nu$, we can take
not only the function $\nu(t)=t$ \cite{Zernov-1} but arbitrary continuous increasing function such that $\nu(0)=0$. This extends the degree of considered singularities. Thus, we can conclude that here some technic improving generalizations of the results from \cite{Zernov-1} are obtained in the linear case.

\subsection{Problem for $k<0$}
Now for $k<0$ we consider the Cauchy problem in the space $\spaceD_-$
with an additional boundary condition at the right end of the interval:
\begin{equation}\label{n-3}
\left\{
\begin{array}{l}
    \dot x(t)=-k\,p(t)\,x(t)+p(t)\nu(t)(Tx)(t)+p(t)\nu(t)f(t),\quad t\in[0,1],\\
    x(0)=0,\quad x(1)=c,
\end{array}
\right.
\end{equation}
where $c\in\spaceR$, $f\in\spaceL_\infty$, $T:\spaceC\to\spaceL_\infty$ is a linear bounded operator.

Problem \eqref{n-3} is equivalent to the equation in the space  $\spaceC$:
\begin{equation*}%\label{rr-50}
  x(t)=(ATx)(t)+(Af)(t)+c\,x_1(t),\quad t\in(0,1],
\end{equation*}
where
\begin{equation*}
  (Az)(t)=-\int_t^1\frac{x_1(t)}{x_1(s)}p(s)\nu(s)z(s)\,ds,\quad t\in(0,1].
\end{equation*}
Estimate the norm of the operator $AT:\spaceC\to\spaceC$:
\begin{equation*}
  \norm{AT}_{\spaceC\to\spaceC}< \nu(1)\frac{\norm{T}_{\spaceC\to\spaceLi}}{|k|}.
\end{equation*}
Therefore, the condition
\begin{equation*}%\label{rr-70}
    \norm{T}_{\spaceC\to\spaceLi}\le \frac{|k|}{\nu(1)}
\end{equation*}
guarantees the contraction of $AT$ and the unique solvability of  \eqref{n-3}.

For a Volterra operator $T:\spaceC\to\spaceLi$, consider the problem
\begin{equation}\label{n-3a}
\left\{
\begin{array}{l}
    \dot x(t)=-k\,p(t)\,x(t)+p(t)\nu(t)(Tx)(t)+p(t)\nu(t)f(t),\quad t\in[0,\alpha],\\
    x(0)=0,\quad x(\alpha)=c,
\end{array}
\right.
\end{equation}
in the space  $\spaceD_-[0,\alpha]$ of all restrictions of all functions from $\spaceD_-$ on $[0,\alpha]$ for  $\alpha\in(0,1]$.

Problem \eqref{n-3a} is equivalent to the equation in the space $\spaceC[0,\alpha]$:
\begin{equation*}%\label{rr-50a}
  x(t)=(A_\alpha Tx)(t)+(A_\alpha f)(t)+c\,x_1(t),\quad t\in(0,\alpha],
\end{equation*}
where
\begin{equation*}
  (A_\alpha z)(t)=-\int_t^\alpha\frac{x_\alpha(t)}{x_\alpha(s)}p(s)\nu(s)z(s)\,ds,\quad t\in(0,1],
\end{equation*}
$x_\alpha(t)=e^{-|k|\int_t^\alpha p(s)\,ds}$.
Estimate the norm of the operator $A_\alpha T:\spaceC[0,\alpha]\to\spaceC[0,\alpha]$:
\begin{equation*}
  \norm{A_\alpha T}_{\spaceC[0,\alpha]\to\spaceC[0,\alpha]}< \nu(\alpha)\frac{\norm{T}_{\spaceC\to\spaceLi}}{|k|}.
\end{equation*}
Therefore, for small enough $\alpha>0$  (such that
   $\nu(\alpha) \norm{T}_{\spaceC\to\spaceLi}\le |k|$)
the operator  $A_\alpha T$ is a contraction, and problem \eqref{n-3a} is uniquely solvable for such $\alpha$.

It turns out, that for the solvability of problem
\eqref{n-2}, generally speaking, the term with the operator $T$ can not contain improving factor $\nu(t)$.

Suppose $T:\spaceC\to\spaceLi$ is a regular operator such that
\begin{equation}\label{rr-80}
    \vraisup_{t\in[0,1]}\frac{(|T|\nu)(t)}{\nu(t)}\equiv\calT<+\infty,
\end{equation}
where
$$T=T^+-T^-,\quad |T|=T^++T^-,$$
$T^+$, $T^-:\spaceC\to\spaceLi$ are linear positive operators.
Any Volterra operator satisfies  \eqref{rr-80}. In this case we have
\begin{equation*}
    \vraisup_{t\in[0,1]}\frac{(|T|\nu)(t)}{\nu(t)}\le
    \vraisup_{t\in[0,1]}\frac{\nu(t)(|T|\unit)(t)}{\nu(t)}=\norm{|T|}_{\spaceC\to\spaceLi}<+\infty.
\end{equation*}
Let condition  \eqref{rr-pnu} be fulfilled.

Consider for the Cauchy problem
\begin{equation}\label{n-2000}
\left\{
\begin{array}{l}
    \dot x(t)=-k\,p(t)\,x(t)+p(t)\,(Tx)(t)+p(t)\,\nu(t)\,f(t),\quad t\in[0,\alpha],\\
    x(0)=0,\quad x(\alpha)=c,
\end{array}
\right.
\end{equation}
on the set
$$\Omega_C[0,\alpha]=\{x\in\spaceD_-[0,\alpha]:\ \vraisup_{s\in[0,\alpha]}\frac{|x(t)|}{\nu(t)}\le C\},$$
where $C$ is a positive constant.

Let  $$\spaceC_{C\nu}[0,\alpha]=\{x\in\spaceC[0,\alpha]:\ \vraisup_{s\in[0,\alpha]}\frac{|x(t)|}{\nu(t)}\le C\}.$$
It is clear, that  $\spaceC_{C\nu}[0,\alpha]$ is a Banach space with the norm
\begin{equation*}
    \norm{x}_{\spaceC_{C\nu}[0,\alpha]}=\vraisup_{t\in[0,\alpha]}\frac{|x(t)|}{\nu(t)}.
\end{equation*}

For $x\in\Omega_C[0,\alpha]$ we have
\begin{equation*}
    \vraisup_{t\in[0,\alpha]}     \frac{|(Tx)(t)|}{\nu(t)}\le \calT \norm{x}_{\spaceC_{C\nu}[0,\alpha]}<+\infty,
\end{equation*}
therefore, a solution of  \eqref{n-2000} for a large enough constant $C$ satisfies the equation in the space  $\spaceC_{C\nu}[0,\alpha]$:
\begin{equation}\label{rr-90}
  x(t)=(A_\alpha Tx)(t)+(A_\alpha(\nu f))(t)+c\,x_1(t),\quad t\in(0,1],
\end{equation}
where
\begin{equation*}
  (A_\alpha z)(t)=-\int_t^\alpha\frac{x_\alpha(t)}{x_\alpha(s)}p(s)z(s)\,ds,\quad t\in(0,\alpha].
\end{equation*}

Consider the function
\begin{equation*}
  K_\alpha(t)=  \int_t^\alpha\frac{x_\alpha(s)}{x_\alpha(s)}p(s)\,\nu(s)\,ds,\quad t\in(0,\alpha],\quad \alpha\in(0,1],
\end{equation*}
where $x_\alpha(t)=e^{-|k|\int_t^\alpha p(s)\,ds}$, $t\in[0,\alpha]$. Obviously, that $K_\alpha(\alpha)=0$, moreover,
\begin{equation*}
    \lim_{t\to0+}K_\alpha(t)=\lim_{t\to0+}\frac{-p(t)\nu(t)}{x_\alpha(t)x_\alpha(t)^{-1}(-p(t))}=\lim_{t\to0+}\nu(t)=0.
\end{equation*}
Find the maximum of the function $K_\alpha(t)/\nu(t)$. At the point $t_*$ of the maximum, the equality
\begin{equation*}
  \dot K_\alpha(t_*)\,\nu(t_*)=  K_\alpha(t_*)\,\dot\nu(t_*)
\end{equation*}
is fulfilled. Therefore,
\begin{equation*}
    \frac{1}{|k|-(\dot \nu(t_*)/(p(t_*)\,\nu(t_*))}=K_\alpha(t_*).
\end{equation*}

Suppose that the function $\nu:(0,1]\to\spaceR$ is continuously differentiable and
\begin{equation}\label{rr-lim}
    \lim_{t\to0+}\frac{\dot \nu(t)}{p(t)\nu(t)}=0.
\end{equation}
For example, if $p(t)=1/t$, then $\nu=-1/ln(t/2)$. Note, that in this case $\nu$ satisfies the singular equation $\dot x(t)=x^2(t)/t$, $t\in(0,1]$. If $p(t)=1/t^\mu$, $\mu>1$, then we may set $\nu=t^r$, $r>0$.

Now for any  $\varepsilon>0$ there exists $\alpha>0$ such that
\begin{equation*}
    \max_{t\in[0,\alpha]}K_\alpha(t)\le \frac{1+\varepsilon}{|k|}.
\end{equation*}
Since
\begin{equation*}
  |A_\alpha(\nu f))(t)|+|c\,x_\alpha(t)|\le \nu(t)\left(\frac{\norm{f}_{\spaceLi}}{|k|}+|c|\,\vraisup\limits_{t\in[0,\alpha]}\frac{|x_\alpha(t)|}{\nu(t)}\right),
  \quad t\in(0,1],
\end{equation*}
and for any  $\varepsilon>0$ for small enough $\alpha>0$ the inequality
\begin{equation*}%\label{rr-110}
  |(A_\alpha Tx)(t)|\le \nu(t) \norm{x}_{\spaceC_{C\nu}} \frac{\calT(1+\varepsilon)}{|k|},\quad t\in[0,1],
\end{equation*}
holds for all $x\in\Omega_C$, then for
\begin{equation}\label{rr-120}
    \calT<k
\end{equation}
the operator of the right-hand side of equation \eqref{rr-90} is a contraction for all small enough $\alpha>0$ and  maps the set $x\in\Omega_C$ into itself for all large enough $C$.

Every solution of  \eqref{n-2000} in he space $\spaceD_-[0,\alpha]$ is bounded, therefore, it belongs to the set $\Omega_C[0,\alpha]$ for some large enough $C\in\spaceR$.
We conclude that if conditions \eqref{rr-lim}, \eqref{rr-pnu}, \eqref{rr-80}, and \eqref{rr-120} are fulfilled (that is
\begin{equation}\label{rr-130}
    \calT\equiv \vraisup_{t\in[0,1]}\frac{(|T|\nu)(t)}{\nu(t)}<k),
\end{equation}
then problem  \eqref{n-2000} in the space $\spaceD_-[0,\alpha]$ has a unique solution for small enough  $\alpha>0$. Note, that condition \eqref{rr-130} can be substituted by
\begin{equation*}
    \calT\equiv \limsup_{\beta\to0+}\vraisup_{t\in[0,\beta]}\frac{(|T|\nu)(t)}{\nu(t)}<k.
\end{equation*}

Condition \eqref{rr-130} is essential: we cannot change this inequality with the non-strict inequality. Indeed, if $Tx=kx$, then problem \eqref{n-2000}, generally speaking, has no solutions.

\end{document}